\documentclass{amsart}
\usepackage{amssymb,amsfonts,amsmath,amsthm,cite,color}
\usepackage[a4paper,margin={2cm,3cm}]{geometry}
\usepackage{epsfig}
\usepackage[latin1]{inputenc}
\usepackage{amsmath,tikz}
\usepackage{tabulary}
\usepackage{longtable}
\usepackage{tabu}
\usepackage{stmaryrd}
\usepackage{multirow}
\usepackage{multicol}
\usepackage{longtable}
\usepackage{cases}
\newtheorem{thm}{Theorem}

\newtheorem{lem}[thm]{Lemma}

\newtheorem{ex}[thm]{Example}

\def\clap#1{\hbox to0pt{\hss#1\hss}}

\def\Stepset#1#2#3#4#5#6#7#8#9{%
  \leavevmode\kern-2pt
  \raisebox{-9pt}{%
  \setlength{\unitlength}{.9pt}%
  \begin{picture}(20,20)(-10,-10)
    \put(-5,-5){\clap{$\scriptstyle\ifx1#1\bullet\else\cdot\fi$}}
    \put(0,-5){\clap{$\scriptstyle\ifx1#2\bullet\else\cdot\fi$}}
    \put(5,-5){\clap{$\scriptstyle\ifx1#3\bullet\else\cdot\fi$}}
    \put(-5,0){\clap{$\scriptstyle\ifx1#4\bullet\else\cdot\fi$}}
    \put(0,0){\clap{$\scriptstyle\ifx1#5\bullet\else\cdot\fi$}}
    \put(5,0){\clap{$\scriptstyle\ifx1#6\bullet\else\cdot\fi$}}
    \put(-5,5){\clap{$\scriptstyle\ifx1#7\bullet\else\cdot\fi$}}
    \put(0,5){\clap{$\scriptstyle\ifx1#8\bullet\else\cdot\fi$}}
    \put(5,5){\clap{$\scriptstyle\ifx1#9\bullet\else\cdot\fi$}}
  \end{picture}}\StepsetB}

\def\StepsetB#1#2#3#4#5#6#7#8{%
  \raisebox{-8pt}{%
  \setlength{\unitlength}{.9pt}%
  \begin{picture}(20,20)(-10,-10)
    \put(-5,-5){\clap{$\scriptstyle\ifx1#1\bullet\else\cdot\fi$}}
    \put(0,-5){\clap{$\scriptstyle\ifx1#2\bullet\else\cdot\fi$}}
    \put(5,-5){\clap{$\scriptstyle\ifx1#3\bullet\else\cdot\fi$}}
    \put(-5,0){\clap{$\scriptstyle\ifx1#4\bullet\else\cdot\fi$}}
    \put(5,0){\clap{$\scriptstyle\ifx1#5\bullet\else\cdot\fi$}}
    \put(-5,5){\clap{$\scriptstyle\ifx1#6\bullet\else\cdot\fi$}}
    \put(0,5){\clap{$\scriptstyle\ifx1#7\bullet\else\cdot\fi$}}
    \put(5,5){\clap{$\scriptstyle\ifx1#8\bullet\else\cdot\fi$}}
  \end{picture}}\StepsetC}

\def\StepsetC#1#2#3#4#5#6#7#8#9{%
  \raisebox{-7pt}{%
  \setlength{\unitlength}{.9pt}%
  \begin{picture}(20,20)(-10,-10)
    \put(-5,-5){\clap{$\scriptstyle\ifx1#1\bullet\else\cdot\fi$}}
    \put(0,-5){\clap{$\scriptstyle\ifx1#2\bullet\else\cdot\fi$}}
    \put(5,-5){\clap{$\scriptstyle\ifx1#3\bullet\else\cdot\fi$}}
    \put(-5,0){\clap{$\scriptstyle\ifx1#4\bullet\else\cdot\fi$}}
    \put(0,0){\clap{$\scriptstyle\ifx1#5\bullet\else\cdot\fi$}}
    \put(5,0){\clap{$\scriptstyle\ifx1#6\bullet\else\cdot\fi$}}
    \put(-5,5){\clap{$\scriptstyle\ifx1#7\bullet\else\cdot\fi$}}
    \put(0,5){\clap{$\scriptstyle\ifx1#8\bullet\else\cdot\fi$}}
    \put(5,5){\clap{$\scriptstyle\ifx1#9\bullet\else\cdot\fi$}}
  \end{picture}}\kern-2pt}

\def\Z{\mathbb Z}

\def\S{\mathcal{S}}

\let\set\mathbb

\begin{document}

 \author[Manuel Kauers]{Manuel Kauers\,$^\ast$}
 \author[Rong-Hua Wang]{Rong-Hua Wang\,$^\ast$}
 \address{Institute for Algebra, J. Kepler University Linz, Austria}
 \email{mkauers@algebra.uni-linz.ac.at}
 \email{ronghua.wang@jku.at}
 \thanks{$^\ast$ Supported by the Austrian FWF grant Y464-N18.}

 \title{Lattice Walks in the Octant with Infinite Associated Groups}

 \begin{abstract}
Continuing earlier investigations of restricted lattice walks in~$\set N^3$, we take a closer look at the models with infinite
associated groups. We find that up to isomorphism, only 12 different infinite groups appear, and we establish a connection
between the group of a model and the model being Hadamard.
 \end{abstract}

 \maketitle

\section{Introduction}

Since the classification project for nearest neighbor lattice walk models in the quarter plane, initiated by Bousquet-Melou and
Mishna~\cite{MelouMishna2010}, is largely completed, the analogous question for 3D models in the octand is getting into 
the focus~\cite{BostanBousquetKauersMelczer2015,bacher16,DuHouWang2016}.
Given a stepset $\mathcal{S}\subseteq\{-1,0,1\}^3\setminus\{(0,0,0)\}$, let $f(x,y,z,t)=\sum_{n,i,j,k}a_{i,j,k,n}x^iy^jz^kt^n$
be the generating function which counts the number $a_{i,j,k,n}$ of walks in $\set N^3$ from $(0,0,0)$ to $(i,j,k)$ consisting of $n$ steps
taken from~$\S$. The main question is then: for which choices~$\mathcal{S}$ is the series $f$ D-finite?

For models in 2D, it turns out that the generating function is D-finite if and only if a certain group associated to 
the model is finite, see for example~\cite{fayolle99,MishnaRechnitzer2009,BostanKauers2010,kurkova12,raschel12,BostanRaschelSalvy2014,MelczerMishna2014,courtiel16} and the references given there.
The situation in 3D seems to be more complicated, as evidenced by some models having a finite group that seem to be non-D-finite~\cite{BostanBousquetKauersMelczer2015,bacher16}.
Among the $2^{3^3-1}$ models, there are (up to bijection) 10,908,263 models which have a group associated to them. 
For 10,905,833 of these models, their group has more than 400 elements. It was shown in~\cite{DuHouWang2016} for all the models with at most six steps that
these groups are in fact infinite. Our first result extends this result to the remaining models.

\begin{thm}\label{thm:infinite}
  For all 3D models with a group with more than 400 elements, the group is in fact infinite. 
\end{thm}

Because of space limitations, and since the proof techniques are exactly the same as in~\cite{MelouMishna2010,DuHouWang2016}, we do not give any further details.
We just mention that we used the fixed point method for 10,905,634 models and the valuation method for the 199 models on
which the fixed point method failed.

In this short paper, we have a closer look at these infinite groups.

\section{Infinite Groups Associated to 3D Models}

Recall the definition of the groups~\cite{MelouMishna2010,BostanBousquetKauersMelczer2015}. Given
$S\subseteq\{-1,0,1\}^3\setminus\{(0,0,0)\}$, let $P_{\mathcal{S}}(x,y,z)=\sum_{(i,j,k)\in \mathcal{S}}x^iy^jz^k$. Collecting coefficients
of $x,y,z$, respectively, we can write
\begin{align*}
P_{\S}(x,y,z)& =x^{-1}A_{-}(y,z)+A_{0}(y,z)+xA_{+}(y,z)\\
             & =y^{-1}B_{-}(x,z)+B_{0}(x,z)+yB_{+}(x,z)\\
             & =z^{-1}C_{-}(x,y)+C_{0}(x,y)+zC_{+}(x,y),
\end{align*}
for certain bivariate Laurent polynomials $A_-,A_0,A_+,B_-,B_0,B_+,C_-,C_0,C_+$. Then the group of $\S$, denoted by $G(\S)$, is generated
by the maps
\[
\phi_x(x,y,z)=\left(\frac{1}{x}\frac{A_{-}}{A_{+}},y,z\right)
,\quad
\phi_y(x,y,z)=\left(x,\frac{1}{y}\frac{B_{-}}{B_{+}},z\right)
,\quad
\phi_z(x,y,z)=\left(x,y,\frac{1}{z}\frac{C_{-}}{C_{+}}\right)
\]
under composition. If one of $A_-,A_+,B_-,B_+,C_-,C_+$ is zero, the group is undefined. The stepsets for which this happens are in bijection
with lower dimensional models, and are excluded from consideration for the rest of this paper. 

\begin{ex}
Let $\S_1=\{(-1, -1, -1), (-1, 1, 1), (1, 0, 1), (1, 1, 0)\}$. 
The group of $\S_1$ is infinite by Theorem~\ref{thm:infinite}. 
Another 3D model with infinite group is $\S_2=\{(-1, 0, 0), (1, -1, 1), (1, 0, 1), (1, 1, -1)\}$.
However, the group for $\S_2$ is in some sense ``less infinite'', because the group generators
satisfy the equations $(\phi_x\phi_y)^2=(\phi_x\phi_z)^2=1$.
There are apparently no such relations for the group of $\S_1$. 
\end{ex}

The examples above suggest that not all infinite groups are equal. This is different from the situation in 2D, where the only possible
infinite group is the infinite dihedral group. In order to understand which groups arise in 3D, we have made a systematic search for relations
among the group generators. According to our computations, there are only the groups listed in Table~\ref{tab}.
\def\a{\mathtt{a}}\def\b{\mathtt{b}}\def\c{\mathtt{c}}\def\<#1>{\langle#1\rangle}%
\begin{table}
  \begin{tabular}{lr|lr}
    Group & \llap{Number of models} & Group &\llap{Number of models} \\\hline
    $G_1=\<\a,\b,\c\mid\a^2,\b^2,\c^2>$ &\kern-1em 10,759,449 & $G_7=\<\a,\b,\c\mid\a^2,\b^2,\c^2,(\a\b)^4>$ & 82 \\
    $G_2=\<\a,\b,\c\mid\a^2,\b^2,\c^2,(\a\b)^2>$ & 84,241 & $G_8=\<\a,\b,\c\mid\a^2,\b^2,\c^2,(\a\b)^3, (\b\c)^3>$ & 30 \\
    $G_3=\<\a,\b,\c\mid\a^2,\b^2,\c^2,(\a\c)^2, (\a\b)^2>$ & 58,642 & $G_9=\<\a,\b,\c\mid\a^2,\b^2,\c^2,\a\c\b\a\c\b\c\a\b\c>$ & 20 \\
    $G_4=\<\a,\b,\c\mid\a^2,\b^2,\c^2,(\a\c)^2, (\a\b)^3>$ & 1,483 & $G_{10}=\<\a,\b,\c\mid\a^2,\b^2,\c^2, (\a\b)^3, (\c\b\c\a)^2>$ & 8 \\
    $G_5=\<\a,\b,\c\mid\a^2,\b^2,\c^2,(\a\b)^3>$ & 1,426 & $G_{11}=\<\a,\b,\c\mid\a^2,\b^2,\c^2,(\c\a)^3,(\a\b)^4, (\b\a\b\c)^2>$ & 8\\
    $G_6=\<\a,\b,\c\mid\a^2,\b^2,\c^2,(\a\c)^2, (\a\b)^4>$ & 440 & $G_{12}=\<\a,\b,\c\mid\a^2,\b^2,\c^2,(\a\b)^4, (\a\c)^4>$ & 4
  \end{tabular}
  \caption{Groups associated to 3D models}\label{tab}
\end{table}
Often, the group generators $\a,\b,\c$ are just the group generators $\phi_x,\phi_y,\phi_z$, but for some models, we need to
change their order or apply simple substitutions such as $\a=\phi_x,\b=\phi_y,\c=\phi_x\phi_z\phi_x$ in order to match their group to one
of the groups listed in Table~\ref{tab}. We must also remark that the relations listed above only are those that we found, and in principle some of the groups
might have further relations. Our systematic search implies that any further relation would correspond to a word of more than 400 generators,
and we are quite confident that no such relations exist. However, proving the absence of additional relations is not an easy thing to do in
general. We consider the two cases which are, in a sense, closest to the case of finite groups. 

\section{The smallest infinite group}\label{sec:HadamardGroup}

We consider the models whose group is isomorphic to~$G_3$. The defining relations of this
group can be read as rewrite rules $\a^2\to\epsilon$, $\b^2\to\epsilon$, $\c^2\to\epsilon$, $\a\c\to\c\a$, $\a\b\to\b\a$. With this rewriting
system, every group element can be written (uniquely) in a form that matches the regular expression $[\c](\b\c)^\ast[\a]$.
If for any of the groups associated to the 58,642 models had an additional relation, we could also write it in this form.
Any such relation however would turn the group into a finite group. Since we know from Theorem~\ref{thm:infinite} that the groups
are infinite, we can exclude the existence of additional equations in this particular case.

Hadamard models were introduced in \cite{BostanBousquetKauersMelczer2015}. They are interesting because their generating function can
be expressed as Hadamard product of the generating functions of two lower dimensional models, and this makes it easier to recognize whether such a model is D-finite. Recall from~\cite{BostanBousquetKauersMelczer2015} that a model is called $(1,2)$-Hadamard if (possibly after a
permutation of variables) its stepset polynomial $P_{\S}$ can be written as
\[
  P_{\S}=U(x)+V(x)T(y,z),
\]
for some $U,V\in\set Q[x,x^{-1}]$ and some $T\in\set Q[y,y^{-1},z,z^{-1}]$. It is called $(2,1)$-Hadamard if (possibly after a permutation
of variables) we have
\[
P_{\S}=U(y,z)+V(y,z)T(x)
\]
for some $U,V\in\set Q[y,y^{-1},z,z^{-1}]$ and some $T\in\set Q[x,x^{-1}]$. 

In the remainder of this section, we establish a connection between the group $G_3$ and Hadamard walks.

\begin{lem}\label{lm:derivative}
  If $f(x,z), g(y,z) \in \mathbb{Q}(x,y,z)$ are such that $f(x,z)=f(\frac{1}{x}g(y,z),z)$, then
  $\frac{\partial}{\partial{x}}f(x,z)=0$ or $\frac{\partial}{\partial{y}}g(y,z)=0$.
\end{lem}
\begin{proof}
If $\frac{\partial}{\partial{x}}f(x,z)\neq0$, then
\[
0=\frac{\partial}{\partial{y}}f(x,z)
 =\frac{\partial}{\partial{y}}f\Big(\frac{1}{x}g(y,z),z\Big)
=(D_1 f)\Big(\frac{1}{x}g(y,z),z\Big)\frac{1}{x}\frac{\partial}{\partial{y}}g(y,z).
\]
Since $\frac{\partial}{\partial{x}}f(x,z)\not=0$, it follows $\frac{\partial}{\partial{y}}g(y,z)=0$, as required.
\end{proof}

\begin{thm}\label{th:HadamardGroup}
  Let $\S$ be a stepset which has an associated group.
  Then $\S$ is Hadamard if and only if $(\phi_x\phi_y)^2=(\phi_x\phi_z)^2=1$
  (possibly after a permutation of the variables $x,y,z$).
\end{thm}
\begin{proof}
  Suppose $\S$ is Hadamard. Then it is easy to check by a direct calculation that we have
  $\phi_x\phi_y=\phi_y\phi_x$ and $\phi_x\phi_z=\phi_z\phi_x$.

For the converse, suppose that $\phi_x\phi_y=\phi_y\phi_x$ and $\phi_x\phi_z=\phi_z\phi_x$.
Then we have 
\begin{align}
\dfrac{A_-(y,z)}{A_+(y,z)}
=\dfrac{A_-\big(\frac{1}{y}\frac{B_-(x,z)}{B_+(x,z)},z\big)}
{A_+\big(\frac{1}{y}\frac{B_-(x,z)}{B_+(x,z)},z\big)}
\qquad\text{and}\qquad
\dfrac{B_-(x,z)}{B_+(x,z)}
=\dfrac{B_-\big(\frac{1}{x}\frac{A_-(y,z)}{A_+(y,z)},z\big)}
       {B_+\big(\frac{1}{x}\frac{A_-(y,z)}{A_+(y,z)},z\big)}\label{eq:1}
\end{align}
and
\begin{align}
\dfrac{A_-(y,z)}{A_+(y,z)}
=\dfrac{A_-\big(y,\frac{1}{z}\frac{C_-(x,y)}{C_+(x,y)}\big)}
{A_+\big(y,\frac{1}{z}\frac{C_-(x,y)}{C_+(x,y)}\big)}
\qquad\text{and}\qquad
\dfrac{C_-(x,y)}{C_+(x,y)}
=\dfrac{C_-\big(\frac{1}{x}\frac{A_-(y,z)}{A_+(y,z)},y\big)}
       {C_+\big(\frac{1}{x}\frac{A_-(y,z)}{A_+(y,z)},y\big)}\label{eq:3}.
\end{align}
If one of $\frac{A_-}{A_+},\frac{B_-}{B_+}$ and $\frac{C_-}{C_+}$ is constant,
e.g. $\frac{A_-}{A_+}=c\not=0$, then
$P_{\S}(x,y,z)=A_{0}(y,z)+A_{-}(y,z)(x^{-1}+cx)$.
Hence $\S$ is a $(2,1)$-Hadamard model and we are done.
If none of $\frac{A_-}{A_+},\frac{B_-}{B_+},\frac{C_-}{C_+}$ is constant,
we claim that:
\begin{equation}\label{eq:V1V2}
\frac{\partial}{\partial x}\Big(\frac{B_-(x,z)}{B_+(x,z)}\Big)=
  \frac{\partial}{\partial x}\Big(\frac{C_-(x,y)}{C_+(x,y)}\Big)=0.
\end{equation}
We prove this claim by contradiction.
Assume $\frac{\partial}{\partial x}\Big(\frac{B_-(x,z)}{B_+(x,z)}\Big)\not=0$, Equation~\eqref{eq:1} (right) and Lemma~\ref{lm:derivative} imply
\[
 \frac{\partial}{\partial y}\Big(\frac{A_-(y,z)}{A_+(y,z)}\Big)=0 \text{ and }
 \frac{\partial}{\partial z}\Big(\frac{A_-(y,z)}{A_+(y,z)}\Big)\not=0.
\]
Then \eqref{eq:3} (left) and Lemma~\ref{lm:derivative} force $\frac{C_-}{C_+}$ to be a constant, which is a contradiction.
Therefore $\frac{\partial}{\partial x}\Big(\frac{B_-(x,z)}{B_+(x,z)}\Big)=0$.
A similar reasoning using Equations~\eqref{eq:3} (right) and~\eqref{eq:1} (left) leads to
$\frac{\partial}{\partial x}\Big(\frac{C_-(x,y)}{C_+(x,y)}\Big)=0$, which completes the proof of the claim.
At this stage, we can assume
\[
\begin{cases}
 B_-=v_1(x) b_-(z)\\
 B_+=v_1(x) b_+(z)
\end{cases}
\text{ and }
\quad
\begin{cases}
 C_-=v_2(x) c_-(z)\\
 C_+=v_2(x) c_+(z).
\end{cases}
\]
Therefore
\begin{align}
 P_{\S}(x,y,z) & = B_0(x,z)+v_1(x)\big(b_-(z)\frac{1}{y}+b_+(z)y\big)
                = C_0(x,y)+v_2(x)\big(c_-(y)\frac{1}{z}+c_+(y)z\big)\label{eq:v2}.
\end{align}
Since $\frac{B_-}{B_+}$ is not a constant, $P_{\S}$ must contain a monomial $m(x,y,z)$ involving both $y$ and~$z$.
Then from \eqref{eq:v2}, we know $v_1(x)=v_2(x)=v(x)$ and every monomial of $P_{\S}$ involving
$y$ or $z$ has the form $v(x) t(y,z)$. Hence $P_{\S}$ can be rewritten as
$P_{\S}(x,y,z)=u(x)+v(x) t(y,z)$, i.e., $\S$ is $(1,2)$-Hadamard.
\end{proof}

%
For a given Hadamard model $\S$, Theorem~\ref{th:HadamardGroup} implies that any $w\in G(\S)$ can be written as 
$\phi_x(\phi_y\phi_z)^m$, $\phi_x(\phi_z\phi_y)^m$, $(\phi_y\phi_z)^m$ or $(\phi_z\phi_y)^m$. Therefore, $G(\S)\cong\set Z_2\times D$, where $D$ is a dihedral group, $D$ being infinite if and only if $G(\S)$ is infinite.
Bacher et al. \cite{bacher16} found 60,829 three dimensional Hadamard models, among which 2,187 are with finite groups
$\mathbb{Z}_2\times D_4$, $\mathbb{Z}_2\times D_{6}$ and $\mathbb{Z}_2\times D_{8}$. This is consistent with our result.
The other 58,642 models are exactly the ones corresponding to the group $G_3=\Z\times D_\infty$ of Table~\ref{tab}.

\section{The second smallest infinite group}

If $G_3$ is the smallest infinite group in our list, then $G_4$ is the second smallest group. Already in this case, we are no longer able
to exclude the existence of further relations. However, we do have some partial results in this direction.
Among the 1,483 models with the group under consideration, there are 29 singular models. 
For a 3D model $\S$ to be singular means that at least one of the three projections of $\S$ to the plane is a 2D singular model
(this is just one of several possible non-equivalent ways to define what a singular model is in 3D). 
Next we will show the absence of further relations for all the 29 singular models having the (conjectured) group~$G_4$ 
via the valuation argument.

The valuation of a Laurent series $F(t)$ is the smallest $d$ such that $t^d$ occurs in $F(t)$ with a non-zero coefficient, denoted by $\text{val}(F)$.
Let $t$ be an indeterminate and $x,y,z$ be Laurent series in~$t$, with coefficients in $\mathbb{Q}$, 
of valuations $u,v$ and $w$ respectively.
Then we can define three new transformations according to the valuation
\begin{alignat*}1
\Phi_x(u,v,w)&=(\text{val}\Big(\frac{A_-}{A_+}\Big)-u,v,w),\\
\Phi_y(u,v,w)&=(u,\text{val}\Big(\frac{B_-}{B_+}\Big)-v,w),\text{ and }\\
\Phi_z(u,v,w)&=(u,v,\text{val}\Big(\frac{C_-}{C_+}\Big)-w).
\end{alignat*}
Suppose $G_{u,v,w}(\S)$ is the group generated by $\Phi_x,\Phi_y$ and $\Phi_z$ under composition.
If $G_{u,v,w}(\S)$ does not have any further relations besides those expected from~$G_4$,
then $G(\S)\cong G_4$.

Using a suitable rewriting system, we can show that all elements of $G_4$ can be brought to a form that matches the regular
expression $[[\a]\b]([\a]\c\b)^*[\a][\c]$. 
Thus every element in $G_{u,v,w}(\S)$ can be written to match 
\begin{equation}\label{eq:valuation not 1}
[[\Phi_x]\Phi_y]\big([\Phi_x]\Phi_z \Phi_y\big)^*[\Phi_x][\Phi_z].
\end{equation}
Next, we will show there exists no further relation in~$G_{u,v,w}(\S)$. 
The idea is to find $(u,v,w)\in\mathbb{Z}^3$ with specific properties such that
\[
 \Phi(u,v,w)\not=(u,v,w),
\]
for any $\Phi\in G_{u,v,w}(\S)$.
The reasoning is best explained with an example.

\begin{ex}
  Consider the singular model $\S=\{(-1,-1,1),(0,1,-1),(1,0,1)\}$. 
Suppose $u,v,w$ are the valuations of $x,y,z$ respectively with $w > v > -u > 0$.
Then
\begin{alignat*}3
 \Phi_x\Phi_z\Phi_y(u,v,w)&=(v - 2 w,{}&&{-}u - v + 2 w, -u - 2 v + 3 w) \text{ and }\\
 \Phi_z\Phi_y(u,v,w)&=(u,&&{-}u- v + 2 w, -u - 2 v + 3 w).
\end{alignat*}
As $w > v > -u > 0$, it is easy to check that $-u - 2 v + 3 w> -u - v + 2 w > -(v - 2 w )>0$ and that
$-u - 2 v + 3 w > w, -u - v + 2 w > v, v - 2 w < u$.
Then by similar discussions for $(u, -u - v + 2 w, -u - 2 v + 3 w)$, we find for any 
$\Phi'\in G_{u,v,w}(\S)$ which matches regular expression $\big([\Phi_x]\Phi_z \Phi_y\big)^*$
\begin{equation}\label{eq:original condition holds}
 \Phi'(u,v,w)=(u',v',w'),
\end{equation}
 where $w'>w,v'>v,u'\leq u$ with $w' > v' > -u' > 0$.

If there exist further relations in $G_{u,v,w}(\S)$, then Equation \eqref{eq:original condition holds} and 
\eqref{eq:valuation not 1}
together with the fact that $\Phi_x,\Phi_y,\Phi_z$ are involutions force the existence of $\Phi\in G_{u,v,w}(\S)$ 
such that $\Phi$ matches $\Phi_x\Phi_z[\Phi_x]\Phi_y\big([\Phi_x]\Phi_z \Phi_y\big)^*$ and that
$
\Phi(u,v,w)=(u,v,w),
$
which is impossible since 
\[
\Phi_y(u',v',w')=(u', -u' - v' + 2 w', w')\text{ with }-u' - v' + 2 w' > v'.
\]
At this stage, we have shown that there is no other relation in $G_{u,v,w}(\S)$.
Therefore, the group associated to $\S$ is really~$G_4$.
\end{ex}
         
The above method applies to all 29 singular models, although the conditions for the valuations differ slightly from model to model. 

\section{Conclusion}

We have noted that not all infinite groups associated to octant models are
equal. Instead, assuming the absence of some unreasonably long relations among
the group generators not implied by shorter relations, we can identify twelve
different infinite groups. Some of them are quite frequent while others are
quite rare. We have seen that one of the groups signals that a model has the
Hadamard property of~\cite{BostanBousquetKauersMelczer2015}. This raises the
question whether also the other groups indicate some useful combinatorial
property of the stepset. So far, we have not found any such connection.

Another important question is whether some of the octant models with an infinite
group have nevertheless a D-finite generating function.  In view of the
seemingly non-D-finite generating functions of certain models with finite
group~\cite{BostanBousquetKauersMelczer2015,bacher16}, we must take this
possibility into account. Testing all the $10^7$ models one by one does not seem
computationally feasible, but maybe a reasonable starting point for such a
search will be the models that have an infinite group other than~$G_1$. We have
performed a search for recurrence relations for these models but did not find
any D-finite models so far. We may have been using too little data.


\begin{thebibliography}{10}

\bibitem{bacher16}
A.~Bacher, M.~Kauers, and R.~Yatchak.
\newblock Continued classification of $3${D} lattice models in the positive
  octant.
\newblock In {\em Proceedings of FPSAC'16}, pages 95--106, 2016.

\bibitem{BostanBousquetKauersMelczer2015}
A.~Bostan, M.~Bousquet-M{\'e}lou, M.~Kauers, and S.~Melczer.
\newblock On 3-dimensional lattice walks confined to the positive octant.
\newblock {\em Ann. Comb.}, 20(4):661--704, 2016.

\bibitem{BostanKauers2010}
A.~Bostan and M.~Kauers.
\newblock The complete generating function for {G}essel walks is algebraic.
\newblock {\em Proc. Amer. Math. Soc.}, 138(9):3063--3078, 2010.
\newblock With an appendix by Mark van Hoeij.

\bibitem{BostanRaschelSalvy2014}
A.~Bostan, K.~Raschel, and B.~Salvy.
\newblock Non-{D}-finite excursions in the quarter plane.
\newblock {\em J. Combin. Theory Ser. A}, 121:45--63, 2014.

\bibitem{MelouMishna2010}
M.~Bousquet-M{\'e}lou and M.~Mishna.
\newblock Walks with small steps in the quarter plane.
\newblock {\em Contemp. Math}, 520:1--40, 2010.

\bibitem{courtiel16}
J.~Courtiel, S.~Melczer, M.~Mishna, and K.~Raschel.
\newblock Weighted lattice walks and universality classes.
\newblock Technical Report 1609.05839, ArXiv, 2016.

\bibitem{DuHouWang2016}
D.~Du, Q.-H. Hou, and R.-H. Wang.
\newblock Infinite orders and non-{D}-finite property of $3$-dimensional
  lattice walks.
\newblock {\em Electron. J. Combin.}, 23:P3.38, 2016.

\bibitem{fayolle99}
G.~Fayolle, R.~Iasnogorodski, and V.~Malyshev.
\newblock {\em Random walks in the quarter-plane}, volume~40 of {\em
  Applications of Mathematics (New York)}.
\newblock Springer-Verlag, Berlin, 1999.
\newblock Algebraic methods, boundary value problems and applications.

\bibitem{kurkova12}
I.~Kurkova and K.~Raschel.
\newblock On the functions counting walks with small steps in the quarter
  plane.
\newblock {\em Publ. Math. Inst. Hautes \'Etudes Sci.}, 116:69--114, 2012.

\bibitem{MelczerMishna2014}
S.~Melczer and M.~Mishna.
\newblock Singularity analysis via the iterated kernel method.
\newblock {\em Combin. Probab. Comput.}, 23(5):861--888, 2014.

\bibitem{MishnaRechnitzer2009}
M.~Mishna and A.~Rechnitzer.
\newblock Two non-holonomic lattice walks in the quarter plane.
\newblock {\em Theoret. Comput. Sci.}, 410(38):3616--3630, 2009.

\bibitem{raschel12}
K.~Raschel.
\newblock Counting walks in a quadrant: a unified approach via boundary value
  problems.
\newblock {\em J. Eur. Math. Soc. (JEMS)}, 14(3):749--777, 2012.

\end{thebibliography}

\section*{Appendix}

We list below the models corresponding to the rare groups $G_8,\dots,G_{12}$.
The models corresponding to the other infinite groups can be obtained from the authors. 
Each model is depicted by three arrangements of dots. For example, the diagram \ \Stepset01111000000000000000001010 \
represents the stepset
\[\{(-1,0,-1),(0,0,-1),(0,-1,-1),(1,-1,-1),(0,1,1),(1,0,1)\}.\]
\begin{center}
\begin{tabular}{|c|c|c|c|c|c|c|c|}\hline
\Stepset01001000000000100000001000 
&\Stepset00000100000001100010000000 
&\Stepset00100000000000110000010000 
&\Stepset00001000000101000000000100 
&\Stepset01000000000001100000001000 
&\Stepset01000000000000110000001000 
&\Stepset10000000000001011000010000 
&\Stepset01001000001010000000000001 
\\\hline
 \Stepset01001000001001000000000100 
&\Stepset01001000001000100000001000 
&\Stepset00001100001001000000000100 
&\Stepset00001100000001100010000000 
&\Stepset01000000000010110000001000 
&\Stepset01000000000001100000011000 
&\Stepset00100000000010110000010000 
&\Stepset00100000000010010000010010 
 \\\hline
 \Stepset11011000011001011000011000 
&\Stepset11001100001011011000010001 
&\Stepset01111000001011110000010100 
&\Stepset01101100001110110000110000 
&\Stepset01101100000111110000010010 
&\Stepset01011000001011011000001011 
&\Stepset01001100001011110000100110 
&\Stepset00101101001011110000010100 
 \\\hline
&\Stepset00101100001011011000010110 
&\Stepset00001100001101110000110110 
&\Stepset00000101100101110010110000 
&\Stepset00101100000101110000010110 
&\Stepset01011000000101110000001011 
 &\Stepset01111000000101110000001010 
 &
 \\\hline 
\end{tabular}

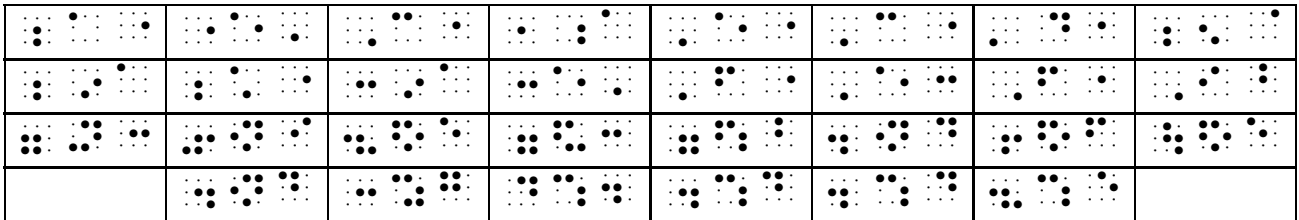
Models with group $G_8=\<\a,\b,\c\mid\a^2,\b^2,\c^2,(\a\b)^3, (\b\c)^3>$

\bigskip

\begin{tabular}{|c|c|c|c|c|c|c|c|}\hline
 \Stepset01010000000100010000010000 
&\Stepset01000100000010001000010000 
&\Stepset00001000001000100000001010 
&\Stepset00001000000100010010100000 
&\Stepset01001000001011100000001000 
&\Stepset01001000000010110000011000 
&\Stepset01001000000001011000110000 
&\Stepset00001100001001110010000000 
 \\\hline
 \Stepset11001100000011001000010001 
&\Stepset01111000000011100000010100 
&\Stepset00101101001000110000010100 
&\Stepset00101000000111000000011110 
&\Stepset11011000001101000000011011 
&\Stepset01111000001110110000000110 
&\Stepset01101100000001011000110110 
&\Stepset00101101000111110000100100 
 \\\hline
&&
 \Stepset01011000001111010000010010 
&\Stepset01001000001011110000011010 
&\Stepset00001101001111010000110000 
 &\Stepset01001100001011011000110000 
&&
 \\\hline 
\end{tabular}

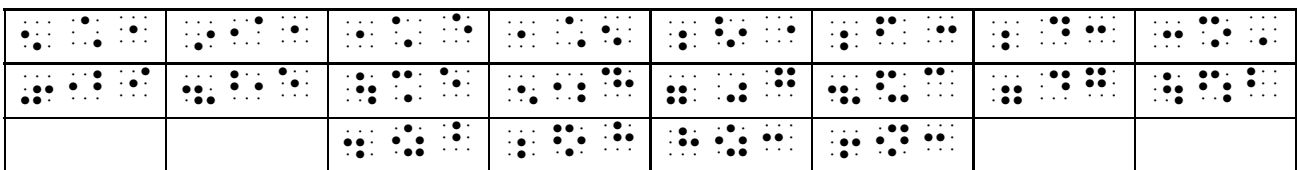
Models with group $G_9=\<\a,\b,\c\mid\a^2,\b^2,\c^2,\a\c\b\a\c\b\c\a\b\c>$

\bigskip
\begin{tabular}{|c|c|c|c|c|c|c|c|}\hline
 \Stepset11000000000001011000010000 
&\Stepset01100000000010110000010000 
&\Stepset01001000001010001000000001 
&\Stepset01001000001001100000000100 
&\Stepset00100100000010110000010000 
&\Stepset00100100000010010000010010 
&\Stepset00001100001001100000000100 
&\Stepset00001100001001000000000110 
 \\\hline
\end{tabular}

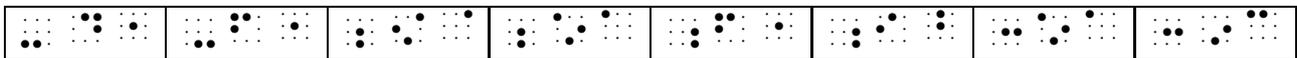
Models with group $G_{10}=\<\a,\b,\c\mid\a^2,\b^2,\c^2, (\a\b)^3, (\c\b\c\a)^2>$
\bigskip

\begin{tabular}{|c|c|c|c|c|c|c|c|}\hline
 \Stepset11000000000000011000010000 
&\Stepset01100000000000110000010000 
&\Stepset01001000000010000000001001 
&\Stepset01001000000001000000100100 
&\Stepset00001100001000000000000110 
&\Stepset00100100000010100000010000 
&\Stepset00100100000010000000010010 
&\Stepset00001000000101000000100100 
  \\\hline
\end{tabular}

Models with group $G_{11}=\<\a,\b,\c\mid\a^2,\b^2,\c^2,(\c\a)^3,(\a\b)^4, (\b\a\b\c)^2>$
\bigskip

\begin{tabular}{|c|c|c|c|}\hline
\Stepset01100100000000000000110010 
&\Stepset01011000000000000000001011 
&\Stepset00101100000000000000010110 
  &\Stepset00000101100000000010110000 
\\\hline
\end{tabular}

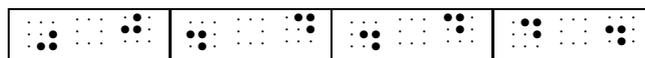
Models with group $G_{12}=\<\a,\b,\c\mid\a^2,\b^2,\c^2,(\a\b)^4, (\a\c)^4>$

\end{center}

\end{document}